\documentclass[11pt]{amsart}

\pdfoutput=1

\usepackage[text={420pt,660pt},centering]{geometry}

\usepackage{cancel}
\usepackage{esint,amssymb}
\usepackage{graphicx}
\usepackage{mathtools} 
\usepackage[colorlinks=true, pdfstartview=FitV, linkcolor=blue, citecolor=blue, urlcolor=blue,pagebackref=false]{hyperref}
\usepackage{microtype}

\usepackage{bm}
\usepackage{dsfont}
\usepackage{mathrsfs}
\usepackage{xcolor}
\usepackage{enumitem} 

\usepackage[normalem]{ulem}

\newtheorem{proposition}{Proposition}
\newtheorem{theorem}[proposition]{Theorem}
\newtheorem{lemma}[proposition]{Lemma}
\newtheorem{corollary}[proposition]{Corollary}

\theoremstyle{remark}
\newtheorem{remark}[proposition]{Remark}

\theoremstyle{definition}

\numberwithin{equation}{section}
\numberwithin{proposition}{section}
\numberwithin{figure}{section}
\numberwithin{table}{section}

\newcommand{\N}{\mathbb{N}}

\newcommand{\R}{\mathbb{R}}

\newcommand{\E}{\mathbb{E}}

\renewcommand{\S}{\mathbf{S}}

\newcommand{\basis}{\Sigma}

\newcommand{\std}{{\mathrm{std}}}
\newcommand{\bx}{{\mathbf{x}}}
\newcommand{\bxi}{{\boldsymbol{\xi}}}
\newcommand{\btheta}{{\boldsymbol{\theta}}}
\newcommand{\identity}{\mathbf{I}}

\newcommand{\eps}{\varepsilon}

\renewcommand{\leq}{\leqslant}
\renewcommand{\geq}{\geqslant}

\renewcommand{\subset}{\subseteq}
\renewcommand{\bar}{\overline}
\renewcommand{\tilde}{\widetilde}

\renewcommand{\hat}{\widehat}

\newcommand{\Ll}{\left}
\newcommand{\Rr}{\right}
\renewcommand{\d}{\mathrm{d}}

\newcommand{\D}{D}
\newcommand{\sP}{\mathscr{P}}

\DeclareMathOperator{\tr}{tr}
\DeclareMathOperator{\supp}{supp}

\newcommand{\la}{\left\langle}
\newcommand{\ra}{\right\rangle}

\newcommand{\diag}{\mathrm{diag}}

\newcommand{\pert}{{\mathrm{pert}}}

\newcommand{\geqpsd}{\succeq} \newcommand{\leqpsd}{\preceq} 
\newcommand{\Sym}{{\mathrm{Sym}(\D)}}
\newcommand{\ks}{{\mathbf{s}}}

\begin{document}

\author{Hong-Bin Chen}
\address{Institut des Hautes \'Etudes Scientifiques, France}
\email{hbchen@ihes.fr}


\title[Self-overlap in vector spin glasses]{On the self-overlap in vector spin glasses}

\begin{abstract}
We consider vector spin glass models with self-overlap correction. Since the limit of free energy is an infimum, we use arguments analogous to those for generic models to show the following: 1) the averaged self-overlap converges; 2) the self-overlap concentrates; 3) the infimum optimizes over paths whose right endpoints are the limit of self-overlap. Lastly, using these, we directly verify the equivalence between the variational formula obtained in \cite{chen2023self} and Panchenko's generalized Parisi formula in \cite{pan.vec}.
\end{abstract}

\maketitle

\section{Introduction}

\subsection{Setting}
Fix $\D\in\N$ and let $P_1$ be a probability measure supported on the unit ball of $\R^\D$. For each $N\in\N$, we consider a probability measure $P_N$ on $\R^{D\times N}$ described as follows.
We denote a sample from $P_N$ by $\sigma = (\sigma_{k,i})_{1\leq k\leq \D,\, 1\leq i\leq N}$. Its $\R^D$-valued column vectors $(\sigma_{\cdot,i})_{1\leq i\leq N}$ are i.i.d.\ with law $P_1$. More precisely, $\d P_N(\sigma) = \otimes_{i=1}^N\d P_1(\sigma_{\cdot,i})$.

For vectors or matrices of the same size, we denote the entry-wise inner product by a dot, e.g., $a\cdot b= \sum_{i,j}a_{i,j}b_{i,j}$. 
We write $|a|=\sqrt{a\cdot a}$. 
We denote the transpose of a matrix $a$ by $a^\intercal$. 
We always evaluate the matrix multiplication before the inner product, e.g. $a\cdot bc = a\cdot(bc)$. We denote by $\S^\D$ (a subset of $\R^{\D\times \D}$) the set of $\D\times\D$ symmetric matrices. We equip $\S^\D$ with the entry-wise inner product. Let $\S^\D_+$ be the subset of $\S^\D$ consisting of positive semi-definite matrices. For any $a,\, b\in \R^{\D\times \D}$, we write $a\geqpsd b$ and $b\leqpsd a$ if $a\cdot c\geq b\cdot c$ for all $c\in \S^\D_+$. 
We often use the following simple fact (c.f.\ \cite[Theorem~7.5.4]{horn2012matrix}) that if $a \in \S^\D$, then
\begin{align}\label{e.a.b>0}
    a\in\S^\D_+ \quad \Longleftrightarrow\quad \forall b\in \S^\D_+,\ a\cdot b\geq 0.
\end{align}

For each $N$, we are given a centered real-valued Gaussian process $(H_N(\sigma))_{\sigma\in\R^{\D\times N}}$ with covariance
\begin{align}\label{e.xi}
    \E H_N(\sigma) H_N(\sigma') = N \xi\Ll(\frac{\sigma\sigma'^\intercal}{N}\Rr),\quad\forall \sigma,\,\sigma'\in \R^{\D\times N}
\end{align}
for some deterministic function $\xi:\R^{\D\times \D}\to \R$. Throughout, we assume that
\begin{enumerate}[start=1,label={\rm (H\arabic*)}]
    \item \label{i.xi_loc_lip}
    $\xi$ is differentiable and $\nabla\xi$ is locally Lipschitz;
    \item \label{i.xi_sym} $\xi\geq 0$ on $\S^\D_+$, $\xi(0)=0$, and $\xi(a)=\xi(a^\intercal)$ for all $a\in \R^{\D\times\D}$;
    \item \label{i.xi_incre} if $a,\,b\in\S^\D_+$ satisfies $a\geqpsd b$, then $\xi(a)\geq \xi(b)$ and $\nabla \xi(a)\geqpsd \nabla \xi(b)$.
\end{enumerate}
The derivative $\nabla \xi:\R^{\D\times \D}\to \R^{\D\times \D}$ is defined with respect to the entry-wise inner product.
The mixed $p$-spin vector spin glass model considered in \cite{pan.vec} belongs to this class.

For each $N\in\N$ and $x\in \S^\D$, we consider the free energy
\begin{align*}
    F_N(x) = \frac{1}{N}\E\log\int \exp\Ll(H_N(\sigma) - \frac{N}{2}\xi\Ll(\frac{\sigma\sigma^\intercal}{N}\Rr)+x\cdot \sigma\sigma^\intercal\Rr)\d P_N(\sigma)
\end{align*}
with the \textit{self-overlap correction} $- \frac{N}{2}\xi\Ll(\frac{\sigma\sigma^\intercal}{N}\Rr) = -\frac{1}{2}\E (H_N(\sigma))^2$ and an external field $x\cdot \sigma\sigma^\intercal$. In the above display, $\E$ integrates the Gaussian randomness in $H_N(\sigma)$.

If $\xi$ is convex on $\S^D_+$, then the limit of free energy is given by the Parisi formula
\begin{align}\label{e.parisi_soc}
    \lim_{N\to\infty} F_N(x) = \inf_{\pi\in\Pi} \sP(\pi,x),\quad\forall x\in \S^\D.
\end{align}
This was proved in \cite[Theorem~1.1]{chen2023self} under a stronger assumption that $\xi$ is convex on $\R^{D\times D}$ which is needed in the Guerra's interpolation for the upper bound. Under the weaker assumption as above, \eqref{e.parisi_soc} was proved in \cite[Corollary~8.3]{HJ_critical_pts} (also see \cite[Remark~8.5]{HJ_critical_pts}), where the upper bound was supplied by the Hamilton--Jacobi equation approach \cite{mourrat2020nonconvex,mourrat2020free,chen2022hamilton,chen2022hamilton2}.
We clarify that~\eqref{e.parisi_soc} is only proved at $x=0$ in both places. We can obtain~\eqref{e.parisi_soc} by substituting $\exp(x\cdot \tau\tau^\intercal)\d P_1(\tau)$ for $P_1$ therein. 
Below, we give definitions related to the right-hand side in~\eqref{e.parisi_soc}.

We set
\begin{align}\label{e.Pi}
    \Pi = \Ll\{\pi:[0,1]\to\S^\D_+\  \big|\  \text{$\pi$ is left-continuous and increasing}\Rr\}
\end{align}
where 
$\pi$ is increasing in the sense that
\begin{align*}
    s'\geq s\quad\implies\quad \pi(s')\geqpsd \pi(s).
\end{align*}
To describe the Parisi functional $\sP$, we need the Ruelle Probability Cascade (RPC). For convenience, we use the continuous version of it following \cite[Section~1]{mourrat2019parisi}. Let $\mathfrak{R}$ be the Ruelle probability cascade with overlap distributed uniformly over $[0,1]$ (see \cite[Theorem~2.17]{pan}). In particular, $\mathfrak{R}$ is a random probability measure on the unit sphere in a separable Hilbert space with the inner product denoted by $\alpha\wedge\alpha'$. Let $\alpha$ and $\alpha'$ be independent samples from $\mathfrak{R}$. Then, the law of $\alpha\wedge\alpha'$ under $\E \mathfrak{R}^{\otimes 2}$ is uniform over $[0,1]$, where $\E$ integrates the randomness in $\mathfrak{R}$. Moreover, almost surely, the support of $\mathfrak{R}$ is ultrametric in the induced topology. For precise definitions and properties, we refer to \cite[Chapter~2]{pan} (see also \cite[Chapter~5]{HJbook}).

For every $\pi\in\Pi$, conditioned on $\mathfrak{R}$, let $(w^\pi(\alpha))_{\alpha\in\supp\mathfrak{R}}$ be a centered $\R^\D$-valued Gaussian process with covariance
\begin{align}\label{e.cov_w^pi}
    \E w^\pi(\alpha)w^\pi(\alpha')^\intercal = \pi (\alpha\wedge\alpha'),\quad\forall \alpha,\,\alpha\in\supp\mathfrak{R}.
\end{align}
We refer to \cite[Section~4 and Remark~4.9]{HJ_critical_pts} for the construction and properties of this process.
By the last property in \ref{i.xi_sym}, we have that $\nabla \xi(a)\in\S^\D$ if $a \in \S^\D$. Then, the first property in \ref{i.xi_incre} and \eqref{e.a.b>0} imply that $\nabla\xi(a)\in\S^\D_+$ for every $a\in\S^\D_+$. 
Using these and the second property in~\ref{i.xi_incre}, we can see that $\nabla\xi\circ\pi\in\Pi$ for every $\pi\in\Pi$. We also need $\theta:\R^{\D\times\D} \to\R$ defined by
\begin{align}\label{e.theta}
    \theta(a) = a\cdot \nabla\xi(a)-\xi(a),\quad\forall a\in \R^{\D\times\D}.
\end{align}
Now, we can define, for $\pi\in\Pi$ and $x\in\S^\D$,
\begin{align}
    \sP(\pi,x)= \E \log\iint \exp\Ll(w^{\nabla\xi\circ\pi}(\alpha)\cdot \tau - \frac{1}{2}\nabla\xi\circ\pi(1)\cdot \tau\tau^\intercal+x\cdot\tau\tau^\intercal\Rr)\d P_1(\tau)\d \mathfrak{R}(\alpha) \label{e.sP(pi,x)}
    \\
    +\frac{1}{2}\int_0^1\theta(\pi(s))\d s .\notag
\end{align}
Here, $\E$ integrates the Gaussian randomness in $w^{\nabla\xi\circ\pi}(\alpha)$ and the randomness of $\mathfrak{R}$.
Note that $- \frac{1}{2}\nabla\xi\circ\pi(1)\cdot \tau\tau^\intercal = \E(w^{\nabla\xi\circ\pi}(\alpha)\cdot \tau)^2$ is the self-overlap correction and $x\cdot \tau\tau^\intercal$ is the external field corresponding to that in $F_N(x)$.

We have completed defining the right-hand side in~\eqref{e.parisi_soc}.

\subsection{Main results}

A very important object in the spin glass theory is the overlap $\frac{\sigma\sigma'^\intercal}{N}$ where $\sigma'$ is an independent copy of $\sigma$ sampled from the Gibbs measure. There has been tremendous progress in understanding the overlap. However, in vector spin glass, not much has been known about a much simpler object, the \textit{self-overlap} $\frac{\sigma\sigma^\intercal}{N}$. Unlike Ising spins or spherical spins, the self-overlap is not constant in general, which causes many difficulties in identifying the limit of the standard free energy without correction
\begin{align}\label{e.F^og_N}
    F^\std_N = \frac{1}{N}\E \log \int \exp\Ll(H_N(\sigma)\Rr)\d P_N.
\end{align}
In particular, we do not know if the self-overlap concentrates or if its average converges.

Our main goal is to answer these questions in the setting with self-overlap correction. Throughout, we set
\begin{align}\label{e.<>_x}
    \la \cdot\ra_x \quad\propto \quad\exp\Ll(H_N(\sigma) - \frac{N}{2}\xi\Ll(\frac{\sigma\sigma^\intercal}{N}\Rr)+x\cdot \sigma\sigma^\intercal\Rr)\d P_N(\sigma)
\end{align}
where the dependence on $N$ is kept implicit. We define $\sP:\S^\D\to\R$ by
\begin{align}\label{e.sP(x)}
    \sP(x) =\inf_{\pi\in\Pi}\sP(\pi,x),\quad\forall x\in \S^\D.
\end{align}
We will show in Proposition~\ref{p.P_diff} that $\sP$ is differentiable and we denote its derivative by $\nabla\sP$ which will be shown to be in $\S^\D_+$ (see~\eqref{e.nabla_sP>0}). For every $z\in\S^\D_+$, we define
\begin{align}\label{e.Pi(z)}
    \Pi(z) = \Ll\{\pi\in\Pi:\:\pi(1)=z\Rr\}.
\end{align}

\begin{theorem}\label{t}
If $\xi$ is convex on $\S^D_+$, then the following holds at every $x\in \S^\D$:
\begin{enumerate}
    \item \label{i.t.cvg_self-overlap} the averaged self-overlap converges:
    \begin{align*}
        \lim_{N\to\infty} \E \la \frac{\sigma\sigma^\intercal}{N}\ra_x = \nabla\sP(x);
    \end{align*}
    \item \label{i.t.conc_self-overlap} the self-overlap concentrates:
    \begin{align*}
        \lim_{N\to\infty} \E \la \Ll|\frac{\sigma\sigma^\intercal}{N}-\E\la \frac{\sigma\sigma^\intercal}{N} \ra_x\Rr|\ra_x =0.
    \end{align*}
    \item \label{i.t.opt} the Parisi formula optimizes over paths with a fixed endpoint:
    \begin{align*}
        \lim_{N\to\infty} F_N(x) = \inf_{\pi\in\Pi(\nabla \sP(x))}\sP(\pi,x).
    \end{align*}
\end{enumerate}
\end{theorem}

The proof relies on the crucial property that the Parisi formula~\eqref{e.parisi_soc} is an infimum in the setting with self-overlap correction. The argument is similar to the one for the generic model (see \cite[Section~3.7]{pan}). Note that the results hold at $x=0$ (i.e., with no the external field), which is the most interesting case.

Our argument does not work in the standard setting~\eqref{e.F^og_N} because the limit of $F^\std_N$ is a supremum after an infimum (see~\eqref{e.pan.parisi}).

\begin{remark}
The assumption that $\xi$ is convex on $\S^D_+$ in Theorem~\ref{t} and results in Section~\ref{s.pfs} is only used to have~\eqref{e.parisi_soc}. Hence, this assumption can be replaced by~\eqref{e.parisi_soc}. Moreover, if we replace $\S^\D$ in~\eqref{e.parisi_soc} by any open subset $O$ of $\S^\D$, we can retain Theorem~\ref{t} and results in Section~\ref{s.pfs} for $x\in O$.
\end{remark}

For the Potts spin glass with symmetric interaction (described in Section~\ref{s.pott}), we can show that the averaged self-overlap is constant and obtain the following corollary. The symmetry argument for this result is inspired by \cite{bates2023parisi}. Throughout, we denote by $\identity_D$ the $\D\times\D$ identity matrix. 

\begin{corollary}\label{c.potts}
    For a Potts spin glass model with $D$ types of spins and self-overlap correction, if $\xi$ is convex on $\S^D_+$ and symmetric, then the following holds:
    \begin{gather*}
        \lim_{N\to\infty} \E \la \Ll|\frac{\sigma\sigma^\intercal}{N}-\frac{1}{D}\identity_D\Rr|\ra_{0} =0,
        \\
        \lim_{N\to\infty} F_N(0) = \inf_{\pi\in\Pi(\frac{1}{D}\identity_D)}\sP(\pi,0).
    \end{gather*}
\end{corollary}
This result trivially extends to $x=r\identity_D$ for any $r\in\R$, because the external field $r\identity_D\cdot \sigma\sigma^\intercal = rN$ simply factors out. We emphasize that this result requires the self-overlap correction and the argument does not work for the standard Potts spin glass.

\subsection{Equivalence of formulae}
Our secondary goal is to understand the equivalence between different forms of the Parisi formula for $F^\std_N$.
Set
\begin{align}\label{e.mathcalD}
    \mathcal{D} = \overline{\mathrm{conv}}\{\tau\tau^\intercal:\:\tau\in\supp P_1\}
\end{align}
where $\overline{\mathrm{conv}}$ takes the closed convex hull of the latter set. We have
\begin{align}
    \lim_{N\to\infty} F^\std_N 
    &= \sup_{z\in \mathcal{D}}\inf_{y\in\S^\D,\,\pi\in\Pi(z) }\Ll\{\sP(\pi,y)-y\cdot z +\frac{1}{2}\xi(z)\Rr\} \label{e.pan.parisi}
    \\
    & = \sup_{z\in \S^\D_+}\inf_{y\in\S^\D_+,\,\pi\in\Pi }\Ll\{\sP(\pi,y)-y\cdot z +\frac{1}{2}\xi(z)\Rr\}.\label{e.hj.parisi}
\end{align}
The first formula~\eqref{e.pan.parisi} was obtained by Panchenko in \cite{pan.vec} through considering free energy with constraint self-overlap. In Appendix~\ref{appendix}, we explain how to rewrite the original formula there into~\eqref{e.pan.parisi}. 
The second formula first appeared in \cite[Corollary~1.3]{mourrat2020extending} for $D=1$ and was later extended to higher dimensions in \cite{chen2023self}. Knowing~\eqref{e.parisi_soc}, one can obtain~\eqref{e.hj.parisi} by solving a Hamilton--Jacobi equation.

Notice that the two formulae differ only in the range of parameters.
The first formula is more descriptive as the set for optimization is finer. The second formula can be useful due to less restriction. For instance, using the convexity of $\sP$ and $\xi$, we can switch the order of optimizations of $y$ and $z$ in \eqref{e.hj.parisi} to get a simpler expression
\begin{align*}
    \lim_{N\to\infty} F^\std_N =\sup_{y\in\S^\D_+}\inf_{\pi\in\Pi}\Ll\{\sP(\pi,y)-\frac{1}{2}\xi^*(2y)\Rr\}
\end{align*}
where $\xi^*(y) = \sup_{z\in\S^\D_+}\Ll\{z\cdot y-\xi(z)\Rr\}$ (see \cite[Theorem~1.2]{chen2023self}).

Aside from the fact that both formulae in~\eqref{e.pan.parisi} and~\eqref{e.hj.parisi} describe the limit free energy, the equivalence between the two does not seem trivial. In Section~\ref{s.equiv}, we directly verify that the two formulae are equal by using the fact that
\begin{align*}
    \sP(x) = \inf_{\pi\in\Pi}\sP(\pi,x) = \inf_{\pi\in\Pi(\nabla\sP(x))}\sP(\pi,x)
\end{align*}
which is a consequence of~\eqref{e.parisi_soc},~\eqref{e.sP(x)}, and Theorem~\ref{t}~\eqref{i.t.opt}. Since the actual proof is technical, we explain the simple heuristics in Section~\ref{s.equiv.heuristics}.

\subsection{Related works}
Parisi initially proposed the formula for the limit free energy in the Sherrington--Kirkpatrick (SK) model in \cite{parisi79,parisi80}. Guerra proved the upper bound in \cite{gue03} and Talagrand proved the matching lower bound in \cite{Tpaper}.
Panchenko extended the formula to various settings: the SK model with soft spins \cite{pan05}, the scalar mixed $p$-spin model \cite{panchenko2014parisi,pan}, the multi-species model \cite{pan.multi}, and the mixed $p$-spin model with vector spins \cite{pan.potts,pan.vec}.
Mourrat's interpretation of the Parisi formula as the Hopf--Lax formula for a Hamilton--Jacobi equation \cite{mourrat2019parisi} led to the extension of the formula to enriched models \cite{mourrat2020extending}.
In the case of spherical spins, the Parisi formula has been established for the SK model \cite{tal.sph}, the mixed $p$-spin model \cite{chen2013aizenman}, and the multi-species model \cite{bates2022free}.

The idea of adding the self-overlap correction first appeared in the Hamilton--Jacobi approach to spin glasses by Mourrat \cite{mourrat2019parisi,mourrat2020extending,mourrat2020nonconvex,mourrat2020free}. With this insight, \cite{chen2023self} revisited vector spin glasses. For more detail on the Hamilton--Jacobi equation approach to statistical mechanics, we refer to~\cite{HJbook}.

Our argument for Theorem~\ref{t} is a straightforward adaption of that for the generic spin glass model presented in \cite[Section~3.7]{pan}.
In particular, Theorem~\ref{t}~\eqref{i.t.conc_self-overlap} follows from the same approach to obtaining the Ghirlanda--Guerra identities for generic spin glasses models \cite{panchenko2010ghirlanda}. 
We will show that the differentiability of $\sP$ and the concentration of the free energy yield the concentration of the self-overlap. 
This is also reminiscent of \cite{chatterjee2009ghirlanda} which proves that the Ghirlanda--Guerra identities hold at differentiable points.

Similar to the method of adding perturbation to ensure the Ghirlanda--Guerra identities \cite{ghirlanda1998general,talagrand2010construction}, one can introduce extra perturbation to ensure the concentration of the self-overlap. This technique first appeared in \cite{mourrat2020nonconvex,mourrat2020free} (see also \cite[Proposition~3.1]{chen2023self}). Our main result shows that the extra perturbation is in fact not needed in spin glasses with self-overlap correction. 

\subsection{Acknowledgements}
The author thanks Chokri Manai and Simone Warzel for discussions which led to this project. The author is grateful to Simone Warzel for hosting a visit where these discussions took place. The author is grateful to Erik Bates from whom the author learned the idea to determine the self-overlap by using symmetry in the Potts spin glass. This project has received funding from the European Research Council (ERC) under the European Union’s Horizon 2020 research and innovation programme (grant agreement No.\ 757296).

\section{Proofs of main results}\label{s.pfs}

In the following, we denote functions $x\mapsto F_N(x)$ and $x\mapsto \sP(x)$ by $F_N$ and $\sP$, respectively.

\subsection{Preliminaries}

In preparation, we introduce the Hamiltonian with perturbation that appears in the cavity computation via the Aizenman--Sims--Starr scheme. This is needed in the proof of Proposition~\ref{p.inf_Pi(nabal_sP)}.

For each $N$, let $(\tilde H_N(\sigma))_{\sigma\in \R^{D\times N}}$ be a centered real-valued Gaussian process with covariance
\begin{align*}\E\tilde H_N(\sigma)\tilde H_N(\sigma')= (N+1)\xi\Ll(\frac{\sigma\sigma'^\intercal}{N+1}\Rr),\quad\forall \sigma,\,\sigma'\in\R^{\D\times N}.
\end{align*}
This is the Hamiltonian appearing in the cavity computation. Since we do not need the exact expression of the perturbation, we refer to \cite[(3.2)]{chen2023self} for the precise definition. We clarify that $x$ and $x^N$ in \cite{chen2023self} denote perturbation parameters, which should not be confused with $x$ here. To avoid confusion, we denote by $\bx$ the perturbation parameter and by $(\bx^N)_{N\in\N}$ a sequence of perturbation parameters. Each perturbation parameter lives in $[0,3]^{\mathbf{N}}$ where $\mathbf{N}$ is a countably infinite set. Let $H^{\pert,\bx}_N(\sigma)$ be the perturbation Hamiltonian given in \cite[(3.2)]{chen2023self}.

For each $N,\, x,\, \bx$, we set
\begin{align*}
    \tilde H^\bx_N(\sigma) &= \tilde H_N(\sigma) -\frac{1}{2}(N+1)\xi\Ll(\frac{\sigma\sigma^\intercal}{N+1}\Rr)+ N^{-\frac{1}{16}} H^{\pert,\bx}_N(\sigma) ,
    \\
    \tilde F^\bx_N(x) &= \frac{1}{N}\E\log\int\exp\Ll(\tilde H^\bx_N(\sigma)+x\cdot\sigma\sigma^\intercal\Rr)\d P_N(\sigma),
    \\
    \la\cdot\ra_{x,\,\bx} &\propto \exp\Ll(\tilde H^\bx_N(\sigma)+x\cdot\sigma\sigma^\intercal\Rr)\d P_N(\sigma).
\end{align*}
Using a standard interpolation argument, one can show that, for each $x$,
\begin{align}\label{e.F-tildeF}
    \lim_{N\to\infty} \sup_{\bx\in [0,3]^\mathbf{N}}\Ll|F_N(x) - \tilde F_N^\bx(x)\Rr| =0.
\end{align}
In the following, we fix an arbitrary sequence $(\bx^N)_{N\in\N}$ of perturbation parameters. We will only specify this sequence in the proof of Proposition~\ref{p.inf_Pi(nabal_sP)}. For each $x$,~\eqref{e.F-tildeF} implies that $F_N(x)$ and $\tilde F^{\bx^N}_N(x)$ converges to the same limit.
We denote by $\tilde F^{\bx^N}_N$ the function $x\mapsto \tilde F^{\bx^N}_N(x)$.

\begin{lemma}\label{l.prop_F_N}
For every $N$, $F_N$ and $\tilde F^{\bx^N}_N$ are Lipschitz and convex. Moreover,
\begin{align*}
    \sup_{N\in\N} \|F_N\|_\mathrm{Lip},\ \sup_{N\in\N} \Ll\|\tilde F^{\bx^N}_N\Rr\|_\mathrm{Lip}\leq 1.
\end{align*}
Consequently, $\sP$ is Lipschitz and convex.
\end{lemma}
\begin{proof}
Fix any $x,\, y \in \S^\D$, for $r\in[0,1]$, we can compute
\begin{align}\label{e.dF_N(x+ry)}
    \frac{\d}{\d r}F_N(x+ry) = \E \la y\cdot \frac{\sigma\sigma^\intercal}{N}\ra_{x+ry}.
\end{align}
By the assumption on the support of $P_1$, we have that
\begin{align}\label{e.|sigmasigma|<N}
    \Ll|\frac{\sigma\sigma^\intercal}{N}\Rr|\leq \frac{1}{N}\sum_{i=1}^N \Ll|\sigma_{\cdot,i}\sigma_{\cdot,i}^\intercal\Rr|\leq 1,\quad
    \text{a.s. under $P_N$.}
\end{align}
Hence, we have $\Ll|\frac{\d}{\d r}F_N(x+ry)\Rr|\leq |y|$ for every $r$ and thus $|F_N(x+y)-F_N(y)|\leq |y|$. This verifies the Lipschitzness uniform in $N$. 

We differentiate one more time to get
\begin{align}
    \frac{\d^2}{\d r^2}F_N(x+ry) & = N\E \la \Ll(y\cdot \frac{\sigma\sigma^\intercal}{N}\Rr)^2 -\Ll(y\cdot \frac{\sigma\sigma^\intercal}{N}\Rr)\Ll(y\cdot \frac{\sigma'\sigma'^\intercal}{N}\Rr) \ra_{x+ry}\notag
    \\
    & = N\E \la \Ll(y\cdot \frac{\sigma\sigma^\intercal}{N}\Rr)^2 -\la y\cdot \frac{\sigma\sigma^\intercal}{N}\ra_{x+ry}^2 \ra_{x+ry}.\label{e.d^2F_N}
\end{align}
Since the right-hand side is nonnegative, we have verified that $F_N$ is convex. The same works for $\tilde F^{\bx^N}_N$.

Lastly, since $F_N$ converges to $\sP$ pointwise due to~\eqref{e.parisi_soc}, we conclude that $\sP$ is also Lipschitz and convex.
\end{proof}

We note that the computation~\eqref{e.dF_N(x+ry)} implies that, for every $x\in\S^\D$,
\begin{align}\label{e.nabla_F=self-overlap}
    \nabla F_N(x) = \E \la \frac{\sigma\sigma^\intercal}{N} \ra_x,\qquad \nabla \tilde F^{\bx^N}_N(x) = \E \la \frac{\sigma\sigma^\intercal}{N} \ra_{x,\,\bx^N}.
\end{align}

\begin{lemma}\label{l.der_P}
For each $\pi$, $\sP(\pi,\cdot)$ is twice differentiable and
\begin{align*}
    \sup_{\substack{\pi,\, x\, \\ |y|\leq 1}}\sum_{i=1}^2\Ll|\frac{\d^j}{\d r^j}\sP(\pi,x+ry)\Big|_{r=0}\Rr|\leq 3.
\end{align*}
\end{lemma}

\begin{proof}
The derivatives of $\sP(\pi,\cdot)$ have the same forms as in \eqref{e.dF_N(x+ry)} and \eqref{e.d^2F_N} but with $N=1$ and a different Gibbs measure. Again, using \eqref{e.|sigmasigma|<N}, we can get the desired estimate.
\end{proof}

\subsection{Differentiability}

\begin{proposition}\label{p.P_diff}
If $\xi$ is convex, then $\sP$ is differentiable everywhere on $\S^\D$ and
\begin{align*}
    \lim_{N\to\infty}\nabla F_N(x) = \lim_{N\to\infty} \nabla\tilde F^{\bx^N}_N (x)= \nabla \sP(x),\quad\forall x\in \S^\D.
\end{align*}
\end{proposition}

By \cite[Theorem~25.5]{rockafellar1970convex}, since $\sP$ is convex and differentiable everywhere, $\sP$ is in fact continuously differentiable.

\begin{proof}
Fix any $x \in \S^\D$. Since $\sP$ is convex by Lemma~\ref{l.prop_F_N}, it suffices to show that any subdifferential $a$ of $\sP$ at $x$ is unique. 
For each $n$, we choose $\pi_n$ to satisfy
\begin{align}\label{e.P(pi_n,x)<P(x)+1/n}
    \sP(\pi_n,x) \leq \sP(x) +\frac{1}{n}.
\end{align}
Fix any $y\in \S^\D$ and let $r\in (0,1]$.
Using the definition of subdifferential, the fact that $\sP$ is an infimum, and \eqref{e.P(pi_n,x)<P(x)+1/n}, we have that
\begin{align*}
    y\cdot a &\leq \frac{\sP(x+ry) - \sP(x)}{r} \leq \frac{\sP(\pi_n,x+ry)-\sP(\pi_n,x)+\frac{1}{n}}{r},
    \\
    y\cdot a &\geq \frac{\sP(x) - \sP(x-ry)}{r} \geq \frac{\sP(\pi_n,x)-\sP(\pi_n,x-ry)-\frac{1}{n}}{r}.
\end{align*}
Using the Taylor expansion of $\sP(\pi_n,\cdot)$ at $x$ and the uniform bound on the second-order derivatives in Lemma~\ref{l.der_P}, we obtain from the above that
\begin{align*}
    \Ll|y\cdot a - y\cdot \nabla\sP(\pi_n,x)\Rr| \leq C r+\frac{1}{nr}
\end{align*}
for some constant $C$ independent of $r$ and $n$. Setting $r= n^{-\frac{1}{2}}$ and sending $n\to\infty$, we can see that $y\cdot a$ is uniquely determined. Since $y$ is arbitrary, we conclude that the subdifferential of $\sP$ at $x$ is unique and thus $\sP$ is differentiable at $x$.

Then, we show the second part.
Let $y$ and $r$ be taken similarly.
The convexity of $F_N$ in Lemma~\ref{l.prop_F_N} implies that
\begin{align*}
    \frac{F_N(x)-F_N(x-ry)}{r} \leq y\cdot\nabla F_N(x) \leq \frac{F_N(x+ry)-F_N(x)}{r}.
\end{align*}
Sending $N\to\infty$ and then $r\to0$, from~\eqref{e.parisi_soc} and the differentiability of $\sP$, we deduce
\begin{align*}
    \lim_{N\to\infty} y\cdot\nabla F_N(x) = y\cdot\nabla \sP(x).
\end{align*}
Since $y$ is arbitrary, we conclude the desired convergence of $\nabla F_N(x)$. The same argument works for $\tilde F^{\bx^N}_N$.
\end{proof}

\begin{proposition}\label{p.so_concent}
If $\xi$ is convex on $\S^D_+$, then, for each $x\in\S^\D$,
\begin{align*}
    \lim_{N\to\infty}\E \la \Ll|\frac{\sigma\sigma^\intercal}{N} - \E \la \frac{\sigma\sigma^\intercal}{N}\ra\Rr|  \ra =0
\end{align*}
holds for $\la \cdot\ra=\la\cdot\ra_x$ and $\la\cdot\ra=\la\cdot\ra_{x,\,\bx^N}$.
\end{proposition}
\begin{proof}
We only demonstrate the argument for $\la\cdot\ra=\la\cdot\ra_x$ and the other case follows verbatim. 

Fix any $y\in \S^\D$ and set $g(\sigma) = y\cdot \sigma\sigma^\intercal$. Since $y$ is arbitrary, it suffices to show that
\begin{align}\label{e.E<|g-E<g>|>=0}
    \lim_{N\to\infty} \frac{1}{N}\E \la \Ll|g(\sigma) - \E \la g(\sigma)\ra_x\Rr|  \ra_x =0.
\end{align}
First, we show that
\begin{align}\label{e.E<|g-<g>|>=0}
    \lim_{N\to\infty} \frac{1}{N}\E \la \Ll|g(\sigma) - \la g(\sigma)\ra_x\Rr|  \ra_x =0.
\end{align}
We denote by $(\sigma^l)_{l\in\N}$ independent copies of $\sigma$ under $\la \cdot\ra_x$. Let $r>0$. Integrating by parts, we have
\begin{align*}
    r\E \la \Ll|g\Ll(\sigma^1\Rr) -g\Ll(\sigma^2\Rr)\Rr|\ra_x = \int_0^r \E \la \Ll|g\Ll(\sigma^1\Rr) -g\Ll(\sigma^2\Rr)\Rr|\ra_{x+sy} \d s
    \\
    -\int_0^r \int_0^t \frac{\d}{\d s} \E \la \Ll|g\Ll(\sigma^1\Rr) -g\Ll(\sigma^2\Rr)\Rr|\ra_{x+sy}\d s \d t.
\end{align*}
We can compute 
\begin{align*}
    \frac{\d}{\d s} \E \la \Ll|g\Ll(\sigma^1\Rr) -g\Ll(\sigma^2\Rr)\Rr|\ra_{x+sy}  =  \E \la \Ll|g\Ll(\sigma^1\Rr) -g\Ll(\sigma^2\Rr)\Rr|\Ll(g\Ll(\sigma^1\Rr)+g\Ll(\sigma^2\Rr)-2g\Ll(\sigma^3\Rr)\Rr)\ra_{x+sy}
    \\
    \geq -2 \E \la \Ll|g\Ll(\sigma^1\Rr) -g\Ll(\sigma^2\Rr)\Rr|^2\ra_{x+sy}\geq -8 \E \la \Ll|g\Ll(\sigma\Rr) -\la g\Ll(\sigma\Rr)\ra_{x+sy}\Rr|^2\ra_{x+sy}.
\end{align*}
Combining the above two displays, we get
\begin{align*}
    \E \la \Ll|g\Ll(\sigma^1\Rr) -g\Ll(\sigma^2\Rr)\Rr|\ra_x &\leq \frac{1}{r}\int_0^r\E \la \Ll|g\Ll(\sigma^1\Rr) -g\Ll(\sigma^2\Rr)\Rr|\ra_{x+sy} \d s
    \\
    &+ \frac{8}{r}\int_0^r \int_0^t \E \la \Ll|g\Ll(\sigma\Rr) -\la g\Ll(\sigma\Rr)\ra_{x+sy}\Rr|^2\ra_{x+sy}\d s \d t
    \\
    &\leq \frac{2}{r}\int_0^r\E \la \Ll|g\Ll(\sigma\Rr) -\la g\Ll(\sigma\Rr)\ra_{x+sy}\Rr|\ra_{x+sy} \d s
    \\
    &+ 8\int_0^r \E \la \Ll|g\Ll(\sigma\Rr) -\la g\Ll(\sigma\Rr)\ra_{x+sy}\Rr|^2\ra_{x+sy}\d s .
\end{align*}
Setting
\begin{align*}
    \eps_N = \frac{1}{N} \int_0^r \E \la \Ll|g\Ll(\sigma\Rr) -\la g\Ll(\sigma\Rr)\ra_{x+sy}\Rr|^2\ra_{x+sy}\d s,
\end{align*}
we thus have
\begin{align*}
    \frac{1}{N}\E \la \Ll|g\Ll(\sigma^1\Rr) -g\Ll(\sigma^2\Rr)\Rr|\ra_x \leq 2 \sqrt{\frac{\eps_N}{rN}}+ 8 \eps_N.
\end{align*}
Note that, by \eqref{e.d^2F_N},
\begin{align*}
    \eps_N & = \int_0^r \frac{\d^2}{\d s^2}F_N(x+sy) \d s =  y\cdot\nabla F_N(x+ry) - y \cdot\nabla F_N(x) 
    \\
    &\leq \frac{F_N(x+(r+t)y) - F_N(x+ry)}{t} - \frac{F_N(x) - F_N(x-t y)}{t}
\end{align*}
for any $t>0$, where the last inequality follows from the convexity of $F_N$.
Therefore, from the above two displays, we obtain
\begin{align*}
    \limsup_{N\to\infty} \frac{1}{8N}\E \la \Ll|g\Ll(\sigma^1\Rr) -g\Ll(\sigma^2\Rr)\Rr|\ra_x \leq \frac{\sP(x+(r+t)y) - \sP(x+ry)}{t} - \frac{\sP(x) - \sP(x-t y)}{t}
\end{align*}
We first send $r\to0$ and then $t\to 0$. By the differentiability of $\sP$, the right-hand side becomes zero. This immediately implies~\eqref{e.E<|g-<g>|>=0} 

Next, we show
\begin{align}\label{e.E|<g>-E<g>|=0}
    \lim_{N\to\infty} \frac{1}{N}\E\Ll| \la g(\sigma)\ra_x - \E\la g(\sigma)\ra_x\Rr|  =0.
\end{align}
We set
\begin{align*}
    \phi_N(x) = \frac{1}{N}\log \int \exp\Ll(H_N(\sigma) - \frac{N}{2}\xi\Ll(\frac{\sigma\sigma^\intercal}{N}\Rr) + x\cdot \sigma\sigma^\intercal\Rr) \d P_N(\sigma).
\end{align*}
Hence, $\E \phi_N(x) =F_N(x)$ and
\begin{align}\label{e.1/NE|g-g|=...}
     \frac{1}{N}\E\Ll| \la g(\sigma)\ra_x - \E\la g(\sigma)\ra_x\Rr| = \E \Ll| y\cdot\nabla \phi_N(x) - y\cdot\nabla F_N(x)\Rr|.
\end{align}
For $r>0$, we set
\begin{align*}
    \delta_N(r) = \Ll|\phi_N(x-ry)-F_N(x-ry)\Rr| + \Ll|\phi_N(x)-F_N(x)\Rr| + \Ll|\phi_N(x+ry)-F_N(x+ry)\Rr|.
\end{align*}
We can compute that the second-order derivative of $\phi_N$ is equal to~\eqref{e.d^2F_N} without $\E$, which is still nonnegative. Hence, $\phi_N$ is convex and we have
that $y\cdot\nabla \phi_N(x) - y\cdot\nabla F_N(x)$ is bounded from above by
\begin{align*}
    \frac{\phi_N(x+ry) - \phi_N(x)}{r} - y\cdot \nabla F_N(x) \leq \frac{F_N(x+ry)-F_N(x)}{r}-y\cdot \nabla F_N(x)+ \frac{\delta_N(r)}{r}
\end{align*}
and from below by
\begin{align*}
    \frac{\phi_N(x) - \phi_N(x-ry)}{r} - y\cdot \nabla F_N(x) \geq \frac{F_N(x)-F_N(x-ry)}{r}-y\cdot \nabla F_N(x)- \frac{\delta_N(r)}{r}.
\end{align*}
By the standard concentration result (e.g.\ see \cite[Theorem~1.2]{pan} or \cite[Theorem~4.7]{HJbook}), there is a constant $C>0$ such that $\sup_{r\in[0,1]}\E \delta_N(r)\leq CN^{-\frac{1}{2}}$.
Using this, \eqref{e.parisi_soc}, and the convergence of $\nabla F_N$ proved in Proposition~\ref{p.P_diff}, we get
\begin{align*}
    \limsup_{N\to\infty} \E \Ll|y\cdot\nabla \phi_N(x) - y\cdot\nabla F_N(x)\Rr| 
    &\leq \Ll|\frac{\sP(x+ry)-\sP(x)}{r}-y\cdot \nabla \sP(x)\Rr| 
    \\
    &+ \Ll|\frac{\sP(x)-\sP(x-ry)}{r}-y\cdot \nabla \sP(x)\Rr|.
\end{align*}
Now, sending $r\to\infty$ and using \eqref{e.1/NE|g-g|=...}, we arrive at~\eqref{e.E|<g>-E<g>|=0}.

Since~\eqref{e.E<|g-<g>|>=0} and~\eqref{e.E|<g>-E<g>|=0} together yield~\eqref{e.E<|g-E<g>|>=0}, the proof is complete.
\end{proof}

\subsection{Upper bound on the path}

\begin{proposition}\label{p.inf_Pi(nabal_sP)}
    If $\xi$ is convex on $\S^D_+$, then
    \begin{align*}
        \lim_{N\to\infty} F_N(x) = \inf_{\pi\in \Pi(\nabla \sP(x))}\sP(\pi,x),\quad\forall x\in \S^\D.
    \end{align*}
\end{proposition}

\begin{proof}
First of all, we verify that
\begin{align}\label{e.nabla_sP>0}
    \nabla\sP(x) \in \S^\D_+,\quad\forall x\in \S^\D.
\end{align}
Due to~\eqref{e.nabla_F=self-overlap}, we can deduce that $\nabla F_N(x)\in \S^\D_+$ for all $N$ and $x$. Since $\S^\D_+$ is a closed set,~\eqref{e.nabla_sP>0} follows from Proposition~\ref{p.P_diff}.

Recall the definition of $\Pi(z)$ in~\eqref{e.Pi(z)}. Display~\eqref{e.nabla_sP>0} makes $\Pi(\nabla \sP(x))$ well-defined.
Since $\Pi(\nabla \sP(x))$ is a subcollection of $\Pi$, we obtain from~\eqref{e.parisi_soc} that
\begin{align*}
    \lim_{N\to\infty} F_N(x) \leq \inf_{\pi\in \Pi(\nabla \sP(x))}\sP(\pi,x).
\end{align*}
It suffices to establish the matching lower bound. Below, we briefly recall the proof of the lower bound for \eqref{e.parisi_soc} (see \cite[Proposition~4.1]{chen2023self}) and explain where we can use Proposition~\ref{p.P_diff} to obtain the desired result.

The proof of \cite[Proposition~4.1]{chen2023self} considers the case $x=0$, but the argument straightforwardly adapts to nonzero cases. For fixed $x\in \S^\D$, we consider the term from the Aizenman--Sims--Starr scheme (see \cite[4.1]{chen2023self})
\begin{align*}
    A_N(\bx) 
    = \E \log\la \int \exp\Ll(Z(\sigma)\cdot \tau - \frac{1}{2}\nabla\xi\Ll(\frac{\sigma\sigma^\intercal}{N}\Rr)\cdot \tau\tau^\intercal\Rr)  \d P_1(\tau)\ra_{x,\,\bx}  \notag
    \\
    - \E \log \la \exp\Ll(Y(\sigma)
    - \frac{1}{2}\theta\Ll(\frac{\sigma\sigma^\intercal}{N}\Rr)\Rr) \ra_{x,\,\bx}\end{align*}
where $(Z(\sigma))_{\sigma\in\R^{\D\times N}}$ and $(Y(\sigma))_{\sigma\in\R^{\D \times N}}$ are, respectively, centered independent $\R^\D$-valued and real-valued Gaussian processes with covariance $\E Z(\sigma)Z(\sigma')^\intercal = \nabla \xi\Ll(\frac{\sigma\sigma'^\intercal}{N}\Rr)$ and $\E Y(\sigma)Y(\sigma')=\theta\Ll(\frac{\sigma\sigma'}{N}\Rr)$.
Using the same argument in the proof of \cite[Lemma~3.3]{pan}, we can find a sequence $(N_k)_{k\in\N}$ of increasing integers and a sequence $\bx^{N_k}$ of perturbation parameters such that
\begin{align}\label{e.liminfF_N}
    \liminf_{N\to\infty}F_N (x)\geq  \lim_{k\to\infty} A_{N_k} \Ll(\bx^{N_k}\Rr)
\end{align}
and following holds.
Set $R^{l,l'}_N = \frac{\sigma^l\Ll(\sigma^{l'}\Rr)^\intercal}{N}$ and $R_N= \Ll(R^{l,l'}_N\Rr)_{l,l'\in\N}$.
Along this subsequence, under $\E \la \cdot\ra_{x,\,\bx^{N_k}}$,
\begin{itemize}
    \item $R_N$ asymptotically satisfies the Ghirlanda--Guerra identities (see \cite[(4.2)]{chen2023self};
    \item there is a random array $R_\infty = \Ll(R^{l,l'}_\infty\Rr)_{l,l'\in\N}$ such that for any $n\in\N$, $\Ll(R_N^{l,l'}\Rr)_{l,l'\leq n}$ converges in distribution to $\Ll(R_\infty^{l,l'}\Rr)_{l,l'\leq n}$.
\end{itemize}
Moreover, Panchenko's synchronization result \cite[Theorem~4]{pan.vec} states that there is a deterministic increasing Lipschitz function $\Psi:[0,\infty)\to \S^\D_+$ such that
\begin{align}\label{e.R=Psi(tr(R))}
    R^{l,l'}_\infty = \Psi\Ll(\tr\Ll(R^{l,l'}_\infty\Rr)\Rr)
\end{align}
almost surely, for every $l,l'\in \N$.

Denote the quantile function of $\tr\Ll(R^{1,2}_\infty\Rr)$ by $\zeta$ (a real-valued left-continuous increasing function on $[0,1]$). Due to $R^{1,2}_N \leqpsd \frac{1}{2}\Ll(R^{1,1}_N + R^{2,2}_N\Rr)$, the concentration in Proposition~\ref{p.so_concent}, and the convergence in Proposition~\ref{p.P_diff}, we have that $R^{1,2}_\infty \leqpsd \nabla \sP(x)$ a.s.\ and thus $\zeta\leq \tr(\nabla \sP(x))$.
Setting $l=l'$ in~\eqref{e.R=Psi(tr(R))}, we have
\begin{align}\label{e.nablasP(x)=Psi}
    \nabla \sP(x)  = \Psi\Ll(\tr\Ll(\nabla\sP(x)\Rr)\Rr).
\end{align}

As explained below \cite[(4.9)]{chen2023self}, we then choose a sequence $(\zeta_m)_{m\in\N}$ (in the same class of functions as $\zeta$) to approximate $\zeta$ and they are chosen to satisfy $\zeta_m(1) = \tr(\nabla \sP(x))$. 
Hence,~\eqref{e.nablasP(x)=Psi} ensures that $\Psi\circ \zeta_m \in \Pi(\nabla \sP(x))$ for all $m$.

Using \eqref{e.R=Psi(tr(R))}, the convergence of $R_N$ to $R_\infty$ and the approximation of $\zeta$ by $\zeta_m$, for any $\eps>0$, one argue as in the paragraph below \cite[(4.10)]{chen2023self} to show
\begin{align*}
    \Ll|A_{N_k}\Ll(\bx^{N_k}\Rr) - \mathscr\sP\Ll(\Psi\circ\zeta_m,x\Rr)\Rr|\leq \eps
\end{align*}
for sufficiently large $k$ and $m$.
This along with~\eqref{e.liminfF_N} implies that
\begin{align*}
    \lim_{N\to\infty} F_N(x) \geq \inf_{\pi \in\in \Pi(\nabla \sP(x))} \sP(\pi, x)-\eps. 
\end{align*}
Sending $\eps\to0$, we obtain the matching lower bound.
\end{proof}

We collect above results to prove the main result.

\begin{proof}[Proof of Theorem~\ref{t}]
    Part~\eqref{i.t.cvg_self-overlap} follows from Proposition~\ref{p.P_diff} and~\eqref{e.nabla_F=self-overlap}; Part~\eqref{i.t.conc_self-overlap} is given by Proposition~\ref{p.so_concent}; and Part~\eqref{i.t.opt} is a restatement of Proposition~\ref{p.inf_Pi(nabal_sP)}.
\end{proof}

\subsection{Application to the Potts spin glass}\label{s.pott}
We start by describing the setting.
Let $\basis=\{e_1,\dots,e_\D\}$ be the standard basis of $\R^\D$. Let $P_1$ be the uniform probability measure on $\basis$. Let $\sigma=(\sigma_{k,i})_{1\leq k\leq \D,\, 1\leq i\leq N}$ be a sample of $P_N$. Due to the column vector $\sigma_{\cdot,i}\in \basis$ for every $i$, we have
\begin{gather}
    \tr\Ll(\sigma\sigma^\intercal\Rr) = \sum_{i=1}^N\sum_{k=1}^\D \sigma_{k,i}\sigma_{k,i} = N,\label{e.tr=N}
    \\
    k\neq k'\quad\implies\quad \Ll(\sigma\sigma^\intercal\Rr)_{k,k'} = \sum_{i=1}^N\sigma_{k,i}\sigma_{k',i} =0.\label{e.kneqk'=0}
\end{gather}
Let $\Sym$ be the permutation group of $\{1,\dots,\D\}$. 
Every element $\Sym$ is a bijection $\ks:\{1,\dots,\D\}\to \{1,\dots,\D\}$.
In addition to the existing assumptions on $\xi$, we assume that
\begin{align}\label{e.xi_sym}
    \xi (a) = \xi\Ll( \Ll(a_{\ks(k),\,\ks(k')}\Rr)_{k,k'}\Rr),\quad\forall a\in\R^{\D\times\D},\ \ks\in \Sym.
\end{align}
An example of $\xi$ satisfying this is the mixed $p$-spin interaction as in \cite[(5)]{pan.vec} with the $p$-th inverse temperature $\beta_p(k) =\beta_p$ independent of $k\in\{1,\dots,\D\}$.

With $P_1$ described above, the model is said to be a \textit{Potts spin glass model}. 
The parameter $D$ is interpreted as the \textit{number of types} of the Potts spin.
We say that $\xi$ is \textit{symmetric} if $\xi$ satisfying~\eqref{e.xi_sym}. These clarify the setting in Corollary~\ref{c.potts}.

\begin{proof}[Proof of Corollary~\ref{c.potts}]
    Using~\eqref{e.xi} and~\eqref{e.xi_sym}, we can show that, for every $\ks$,
\begin{align*}
    \Ll(H_N(\sigma)\Rr)_{\sigma} \stackrel{\d }{=} \Ll(H_N\Ll(\Ll(\sigma_{\ks(k),\,i}\Rr)_{k,i}\Rr)\Rr)_{\sigma}.
\end{align*}
Under $P_N$, we also have $\sigma \stackrel{\d}{=} \Ll(\sigma_{\ks(k),\,i}\Rr)_{k,i}$ for every $\ks$. Using these and the expression of $\la\cdot\ra_0$ in~\eqref{e.<>_x}, we have
\begin{align}\label{e.E<f>=E<fp>}
    \E \la f(\sigma)\ra_0= \E \la f\Ll(\Ll(\sigma_{\ks(k),\,i}\Rr)_{k,i}\Rr)\ra_0
\end{align}
for every bounded continuous function $f:\R^{\D\times N}\to \R$ and every $\ks$.

Combining~\eqref{e.tr=N} and~\eqref{e.E<f>=E<fp>}, we have
\begin{align*}
    \E\la \Ll(\sigma\sigma^\intercal\Rr)_{k,k}\ra_0 = \frac{1}{\D}\sum_{k'=1}^\D \E\la \Ll(\sigma\sigma^\intercal\Rr)_{k',k'}\ra_0 = \frac{N}{\D},\quad\forall k\in\{1,\dots,\D\}.
\end{align*}
This along with~\eqref{e.kneqk'=0} implies that
\begin{align*}
    \E\la \frac{\sigma\sigma^\intercal}{N}\ra_0 = \frac{1}{\D}\identity_\D.
\end{align*}
Corollary~\ref{c.potts} follows from this and Theorem~\ref{t}.
\end{proof}

\section{Equivalence of variational formulae}\label{s.equiv}

We directly verify the equivalence of formulae in~\eqref{e.pan.parisi} and~\eqref{e.hj.parisi}.
\begin{proposition}\label{p.equiv}
If $\xi$ is convex on $\S^D_+$, then
\begin{align}\label{e.equiv_formulae}
    \sup_{z\in\S^\D_+}\inf_{y\in\S^\D_+,\,\pi\in\Pi}\Ll\{\mathscr{P}(\pi,y) -y\cdot z+\frac{1}{2}\xi(z)\Rr\} = \sup_{z\in\mathcal{D}}\inf_{y\in\S^\D,\,\pi\in\Pi(z)}\Ll\{\mathscr{P}(\pi,y) -y\cdot z+\frac{1}{2}\xi(z)\Rr\}.
\end{align}
\end{proposition}
We emphasize that the equivalence is already known since they both describe the limit of $F^\std_N$.
The goal here is to prove the equivalence directly using properties of the Parisi formula and the Parisi functional. 
We want to demonstrate that essentially we only need some elementary analysis and the relation
\begin{align}\label{e.P(x)=infP=infP}
    \sP(x) = \inf_{\pi\in\Pi}\sP(\pi,x) = \inf_{\pi\in\Pi(\nabla \sP(x))}\sP(\pi,x),\quad\forall x\in \S^\D,
\end{align}
which is a consequence of~\eqref{e.parisi_soc},~\eqref{e.sP(x)}, and Theorem~\ref{t}~\eqref{i.t.opt}.
The assumption that $\xi$ is convex on $\S^D_+$ is only used to ensure the second equality in the above display.
In the following, we will use the first equality in \eqref{e.P(x)=infP=infP} repeatedly without mentioning.

\subsection{Heuristic derivation}\label{s.equiv.heuristics}
As explained in \cite[Section~5]{chen2023self}, the formula~\eqref{e.hj.parisi} is obtained by solving some Hamilton--Jacobi equation over $(0,\infty)\times \S^\D_+$, which is the reason for optimizations over $\S^\D_+$ in~\eqref{e.hj.parisi}. First, we claim that we can actually solve the equation over $(0,\infty)\times \S^\D$, which allows us to optimize over $y,\,z\in\S^\D$. Hence, the left-hand side in~\eqref{e.equiv_formulae} is equal to itself with $\S^\D_+$ replaced by $\S^\D$.

Then, we can infer from Theorem~\ref{t}~\eqref{i.t.cvg_self-overlap} that the closed convex hull $K$ of $\{\nabla\sP(x):x\in\S^\D\}$ is contained in $\mathcal{D}$ (defined in~\eqref{e.mathcalD}). 
Therefore, if $z\not\in \mathcal{D}$, then $z\not \in K$ and thus we expect to have
\begin{align*}
    \inf_{y\in\S^\D,\pi\in\Pi}\Ll\{\sP(\pi,y)-y\cdot z\Rr\} = \inf_{y\in\S^\D} \{\sP(y)-y\cdot z\}=-\infty,\quad\forall z\not\in\mathcal{D}.
\end{align*}
Hence,
\begin{align*}
    \sup_{z\in\S^\D}\inf_{y\in\S^\D,\,\pi\in\Pi}\Ll\{\mathscr{P}(\pi,y) -y\cdot z+\frac{1}{2}\xi(z)\Rr\} = \sup_{z\in\mathcal{D}}\inf_{y\in\S^\D,\,\pi\in\Pi}\Ll\{\mathscr{P}(\pi,y) -y\cdot z+\frac{1}{2}\xi(z)\Rr\}.
\end{align*}
For any $z\in \mathcal{D}$, suppose that 
$\inf_{y\in\S^\D,\pi\in\Pi}\Ll\{\sP(\pi,y)-y\cdot z\Rr\}$ is achieved at some $y_z$. Then, we have $\nabla \sP(y_z) = z$ which by~\eqref{e.P(x)=infP=infP} implies that
\begin{align*}
    \inf_{y\in\S^\D,\pi\in\Pi}\Ll\{\sP(\pi,y)-y\cdot z\Rr\} =  \inf_{\pi\in\Pi}\Ll\{\sP(\pi,y_z)-y_z\cdot z\Rr\} = \inf_{\pi\in\Pi(z)}\Ll\{\sP(\pi,y_z)-y_z\cdot z\Rr\} 
    \\
    \geq \inf_{y\in\S^\D,\pi\in \Pi(z)} \Ll\{\sP(\pi,y)-y\cdot z\Rr\} .
\end{align*}
On the other hand, due to $\Pi\supset \Pi(z)$,
\begin{align*}
    \inf_{y\in\S^\D,\pi\in\Pi}\Ll\{\sP(\pi,y)-y\cdot z\Rr\} \leq \inf_{y\in\S^\D,\pi\in \Pi(z)} \Ll\{\sP(\pi,y)-y\cdot z\Rr\} .
\end{align*}
Therefore, the above must be an equality, which then implies 
\begin{align*}
    \sup_{z\in\mathcal{D}}\inf_{y\in\S^\D,\,\pi\in\Pi}\Ll\{\mathscr{P}(\pi,y) -y\cdot z+\frac{1}{2}\xi(z)\Rr\} = \sup_{z\in\mathcal{D}}\inf_{y\in\S^\D,\,\pi\in\Pi(z)}\Ll\{\mathscr{P}(\pi,y) -y\cdot z+\frac{1}{2}\xi(z)\Rr\}
\end{align*}
and the desired representation.

\subsection{Preliminary lemmas}

To obtain a rigorous proof of Proposition~\ref{p.equiv},
we need a few technical lemmas.

\begin{lemma}\label{l.y_eps}
Let $f:\S^\D\to\R$ be differentiable everywhere. Let $z\in \S^\D$ satisfy
\begin{align*}
    \inf_{y\in\S^\D}\{f(y)- y\cdot z\}>-\infty.
\end{align*}
For every $\eps>0$, there is $y_\eps \in \S^\D$ such that
\begin{gather}
     \inf_{y\in\S^\D}\Ll\{f(y) +\eps\sqrt{1+|y|^2}- y\cdot z\Rr\} = f(y_\eps) + \eps\sqrt{1+|y_\eps|^2} - y_\eps\cdot z,\label{e.y_eps_exists}
    \\
    |z-\nabla f(y_\eps)|\leq \eps. \label{e.|z-nabla_f|<eps}
\end{gather}
Moreover,
\begin{gather}
    \lim_{\eps\to0}\inf_{y\in\S^\D}\Ll\{f(y) +\eps\sqrt{1+|y|^2}- y\cdot z\Rr\} = \inf_{y\in\S^\D}\{f(y) - y\cdot z\}. \label{e.liminf<}
\end{gather}

\end{lemma}
\begin{proof}
By assumption, there is a constant $C>0$ such that $f(y)-y\cdot z\geq -C$ for all $y$, which implies that
\begin{align*}
    f(y) +\eps \sqrt{1+|y|^2}-y\cdot z \geq \eps\sqrt{1+|y|^2}  -C.
\end{align*}
Since the right-hand side is unbounded as $|y|\to\infty$, the left-hand side must achieve its minimum in a bounded set, which yields the existence of $y_\eps$ satisfying~\eqref{e.y_eps_exists}.

Using~\eqref{e.y_eps_exists}, we can differentiate at $y_\eps$ to see $z= \nabla f(y_\eps) + \frac{\eps y_\eps}{\sqrt{1+|y_\eps|^2}}$ which yields~\eqref{e.|z-nabla_f|<eps}.

For any $\delta>0$, there is $\bar y$ such that
\begin{align*}
    f(\bar y) - \bar y\cdot z \leq \inf_{y\in\S^\D}\{f(y)-y\cdot z\}+\frac{\delta}{2}.
\end{align*}
Then, we can choose $\bar \eps$ such that, for all $\eps\in(0,\bar\eps)$, 
\begin{align*}
    f(\bar y) +\eps\sqrt{1+|\bar y|^2} - \bar y\cdot z \leq \inf_{y\in\S^\D}\{f(y)-y\cdot z\}+\delta,
\end{align*}
which implies that
\begin{align*}
    \inf_{y\in\S^\D}\{f(y)-y\cdot z\}\leq \inf_{y\in \S^\D}\Ll\{f(y) +\eps\sqrt{1+| y|^2} -  y\cdot z\Rr\} \leq \inf_{y\in\S^\D}\{f(y)-y\cdot z\}+\delta
\end{align*}
Sending $\eps\to0$ and then $\delta\to0$, we obtain~\eqref{e.liminf<}.
\end{proof}

\begin{lemma}\label{l.-infty_y_eps}
Let $f:\S^\D\to\R$ be Lipschitz and differentiable everywhere. Let $z\in \S^\D$ satisfy
\begin{align*}
    \inf_{y\in\S^\D}\{f(y)- y\cdot z\}=-\infty.
\end{align*}
For every $\eps>0$, there is $y_\eps \in \S^\D$ such that
\begin{gather}
     \inf_{y\in\S^\D}\{f(y) +\eps|y|^2- y\cdot z\} = f(y_\eps) + \eps|y_\eps|^2 - y_\eps\cdot z.\label{e.-infty_y_eps_exists}
\end{gather}
Moreover,
\begin{align} \label{e.-infty_liminf<}
    \lim_{\eps\to0}\inf_{y\in\S^\D}\{f(y) +\eps|y|^2- y\cdot z\} =-\infty.
\end{align}
\end{lemma}

\begin{proof}
Note that, for all $y\in\S^\D$,
\begin{align*}
    f(y) +\eps |y|^2-y\cdot z \geq f(0)-\|f\|_\mathrm{Lip}|y|-|z||y| +\eps|y|^2.
\end{align*}
Since the right-hand side is unbounded as $|y|\to\infty$, the infimum of the left-hand side must be achieved at some $y_\eps$, which verifies~\eqref{e.-infty_y_eps_exists}.

By assumption, for any $M>0$, there is $\bar y$ such that
\begin{align*}
    f(\bar y) - \bar y\cdot z \leq -M.
\end{align*}
For sufficiently small $\eps$, we have
\begin{align*}
    f(\bar y) +\eps|\bar y|^2- \bar y\cdot z \leq -M+1,
\end{align*}
which implies~\eqref{e.-infty_liminf<}.
\end{proof}

\begin{lemma}\label{l.Lip_P(pi,x)}
For every $x\in\S^\D$ and every $\pi,\,\pi'\in\Pi$,
\begin{align*}
    \Ll|\sP(\pi,x)-\sP(\pi',x)\Rr|\leq \frac{1}{2}\int_0^1 \Ll(\Ll|\nabla\xi\circ\pi(s)-\nabla\xi\circ\pi'(s)\Rr|+\Ll|\theta\Ll(\pi(s)\Rr)-\theta\Ll(\pi'(s)\Rr)\Rr| \Rr)\d s.
\end{align*}
\end{lemma}
\begin{proof}
We write $\mu=\nabla\xi\circ\pi$ and $\mu'=\nabla\xi\circ\pi'$.
For $r\in[0,1]$, we set $\mu_r = (1-r)\mu+r\mu'$. Let $(w^\mu(\alpha))_{\alpha\in \supp\mathfrak{R}}$ and $(w^{\mu'}(\alpha))_{\alpha\in \supp\mathfrak{R}}$ be independent when conditioned on $\mathfrak{R}$. In view of~\eqref{e.cov_w^pi}, we have
\begin{align*}
    \sqrt{1-r} w^\mu(\alpha) + \sqrt{r}w^{\mu'}(\alpha) \stackrel{\d}{=} w^{\mu_r}(\alpha).
\end{align*}
Hence, in the definition of $\sP(\pi_r,x)$ (see~\eqref{e.sP(pi,x)}), we can replace $w^{\mu_r}(\alpha)$ by $\sqrt{1-r} w^\mu(\alpha) + \sqrt{r}w^{\mu'}(\alpha)$. We write $\sP(\pi_r,x) = f(\pi_r)+ \frac{1}{2}\int_0^1\theta\Ll(\pi_r(s)\Rr)\d s$. Then, we can compute
\begin{align*}
    \frac{\d}{\d r} f(\pi_r) &= \E \la \Ll(\frac{1}{2\sqrt{r}}w^{\mu'}(\alpha) -  \frac{1}{2\sqrt{1-r}}w^{\mu}(\alpha)\Rr)\cdot \tau -\frac{1}{2}\Ll(\mu'(1)-\mu(1)\Rr)\cdot\tau\tau^\intercal \ra
    \\
    & = \frac{1}{2}\E \la \Ll(\mu\Ll(\alpha\wedge\alpha'\Rr)-\mu'\Ll(\alpha\wedge\alpha'\Rr)\Rr)\cdot\tau\tau'^\intercal\ra
\end{align*}
where $\la\cdot\ra$ is the obvious Gibbs measure associated with $f(\pi_r)$ and we used the Gaussian integration by parts (see \cite[Theorem~4.6]{HJbook}) to derive the second equality.
By the invariance property of the Ruelle probability cascade \cite[Theorem~4.4]{pan} (and \cite[Proposition 4.8]{HJ_critical_pts}), $\alpha\wedge\alpha'$ is still distributed uniformly over $[0,1]$ under $\E\la\cdot\ra$. Using this and $\Ll|\tau\tau^\intercal\Rr|\leq 1$ due to the assumption on $P_1$, we get
\begin{align*}
    \Ll|\frac{\d}{\d r} f(\pi_r)\Rr|  \leq \frac{1}{2}\E \la\Ll|\mu(\alpha\wedge\alpha')-\mu'(\alpha\wedge\alpha')\Rr|\ra = \frac{1}{2}\int_0^1 \Ll|\mu(s)-\mu'(s)\Rr|\d s.
\end{align*}
The desired result follows from integration in $r$ from $0$ to $1$.
\end{proof}

We need some notation and simple results for matrices.
For $a \in\S^\D_+$, 
\begin{itemize}
    \item if $a\neq 0$, we denote by $m(a)$ the smallest \textit{positive} eigenvalue of $a$;
    \item by diagonalization, we have $\frac{1}{\sqrt{D}}\tr(a)\leq |a|\leq \tr(a)$ and also that $|a|$ is bounded by $\sqrt{D}$ times its largest eigenvalue;
    \item hence, if $a\in\S^\D_{++}$, then $\Ll|a^{-1}\Rr|\leq \sqrt{D}m(a)^{-1}$;
    \item we denote by $\sqrt{a}$ the matrix square root of $a$ and we have $\Ll|\sqrt{a}\Rr|=\sqrt{\tr(a)}\leq D^\frac{1}{4}\sqrt{|a|}$.
\end{itemize}
Due to $(a+b)\cdot(a-b) = |a|^2-|b|^2$, we have $|a|\geq |b|$ if $a\geqpsd b$.
The Powers--St\o rmer inequality \cite[Lemma~4.1]{powers1970free} gives $\Ll|\sqrt{a+b}-\sqrt{a}\Rr| \leq \sqrt{\tr(b)}$ for all $a,b\in\S^\D_+$.

For $z\in\S^\D_+$ and $r\in[0,\infty)$, we set
\begin{align}\label{e.K(z,r)=}
    K(z,r) =
    \begin{cases}
        (|z|+r)\Ll(1+|z|^\frac{1}{2}m(z)^{-\frac{1}{2}}\Rr)m(z)^{-\frac{1}{2}}r^\frac{1}{2} + \sqrt{(|z|+r)r}, &\quad z \in \S^\D_+\setminus\{0\},
        \\
        r, &\quad z =0.
    \end{cases}
\end{align}
When $z \in \S^\D_+\setminus\{0\}$, $z$ has a least one positive eigenvalue and thus $m(z)>0$ and thus $K(z,r)$ is always finite.

\begin{lemma}\label{e.pipi'}
There is a constant $C$ depending only on $\D$ such that for every $w, z\in\S^\D_+$ and $\pi\in\Pi(z+w)$, there is $\pi'\in \Pi(z)$ such that
\begin{align}\label{e.<CK}
    \int_0^1\Ll|\pi(s)-\pi'(s)\Rr|\d s \leq CK(z,|w|).
\end{align}
\end{lemma}
\begin{proof}
If $z=0$, then $\Pi(z)$ contains only the path constantly equal to zero and \eqref{e.<CK} holds.

First, we consider the case where $z\in\S^\D_{++}$.
We set $h= \sqrt{z}\Ll(\sqrt{z+w}\Rr)^{-1}$ and define $\pi'(s) = h \pi(s)h^\intercal$ for $s\in[0,1]$. It is immediate that $\pi'\in\Pi$. Due to $h(z+w)h^\intercal =z$, we have $\pi'\in \Pi(z)$. Then, we compute
\begin{align*}
    \Ll|\pi-\pi'\Rr|\leq \Ll|\pi-h\pi\Rr|+\Ll|h\pi-h\pi h^\intercal\Rr| \leq |\pi||\identity-h|+|h||\pi||\identity-h|=|\pi|(1+|h|)|\identity-h|,
\end{align*}
Due to $\pi\leqpsd z+w$, we have $|\pi|\leq |z|+|w|$. Due to $m\Ll(\sqrt{z+w}\Rr) = m\Ll(z+w\Rr)^\frac{1}{2}\geq m(z)^\frac{1}{2}$, we can bound
\begin{align*}
    |h|\leq \Ll|\sqrt{z}\Rr|\Ll|\Ll(\sqrt{z+w}\Rr)^{-1}\Rr| \leq \Ll(D^\frac{1}{4}|z|^\frac{1}{2}\Rr)\Ll(\sqrt{D}m\Ll(\sqrt{z+w}\Rr)^{-1}\Rr)\leq D^\frac{3}{4}|z|^\frac{1}{2}m(z)^{-\frac{1}{2}}.
\end{align*}
Similarly,
\begin{align*}
    \Ll|\identity-h\Rr|&=\Ll|\Ll(\sqrt{z+w}-\sqrt{z}\Rr)\Ll(\sqrt{z+w}\Rr)^{-1}\Rr|\leq \sqrt{\tr(w)}\sqrt{D}m\Ll(\sqrt{z+w}\Rr)^{-1}
     \leq  D^\frac{3}{4}m(z)^{-\frac{1}{2}}|w|^\frac{1}{2}.
\end{align*}
Combining these, we get \eqref{e.<CK} if $z\in\S^\D_{++}$.

Next, we consider the general case $z\in\S^\D_{+}\setminus\{0\}$. By applying an orthogonal transformation, we can assume $z = \diag(\lambda_1,\dots,\lambda_\D)$ for $\lambda_1\geq\dots\geq \lambda_\D$. Set $d = \max\{i: \lambda_i>0\}$. For every $a\in\R^{\D\times \D}$, we write $\check a =(a_{ij})_{1\leq i,j\leq d}$. Then, we have $\check z\in\S^d_{++}$, $|\check z|= |z|$, and $m(\check z) = m(z)$. We set $h= \sqrt{\check z}\Ll(\sqrt{\check z +\check w}\Rr)^{-1}$ and $p(s) = h\check \pi(s) h^\intercal $ for $s\in[0,1]$. Then, by the same argument as above,
\begin{align*}
    \int_0^1\Ll|\check\pi(s)-p(s)\Rr|\d s \leq C K\Ll(\check z,\Ll|\check w\Rr|\Rr)\leq C K(z,|w|)
\end{align*}
for some $C$ depending only on $\D$. For $b\in \R^{d\times d}$, we write $\hat b = \diag(b,\mathbf{0})$ where $\mathbf{0}$ is the $(D-d)\times (D-d)$ zero matrix. We define $\pi' = \hat p$. Then,
\begin{align*}
    \int_0^1\Ll|\pi(s)-\pi'(s)\Rr|\d s  \leq  \int_0^1\Ll|\pi(s)-\hat{\check \pi}(s)\Rr|\d s + \int_0^1\Ll|\check\pi(s)-p(s)\Rr|\d s .
\end{align*}
To estimate $\pi-\hat{\check \pi}$, we only need to look at $\pi_{ij}$ with $\max\{i,j\}\geq d+1$. For every $i\in\{1,\dots,\D\}$, we have $\pi_{ii}\leq |\pi|\leq |z|+|w|$. If $i\geq d+1$, then $\pi_{ii} \leq w_{ii}\leq |w|$. 
Since $\pi$ is $\S^\D_+$-valued, we have $|\pi_{ij}|\leq \sqrt{\pi_{ii}\pi_{jj}}$ for every $i,j$.
Hence, if $\max\{i,j\}\geq d+1$, then $|\pi_{ij}|\leq \sqrt{(|z|+|w|)|w|}$. This along with the above displays yields \eqref{e.<CK}.
\end{proof}

\begin{lemma}\label{l.infP>infP}
For every $R>0$, there is a constant $C$ such that for every $y\in\S^\D$ and $z,z'\in \S^\D_{+}$ satisfying $|z|,\,|z'|\leq R$,
\begin{align*}
    \inf_{\pi\in\Pi(z')}\sP(\pi,y) - \inf_{\pi\in\Pi(z)}\sP(\pi,y) \geq- CK\Ll(z,\Ll|z-z'\Rr|\Rr).
\end{align*}
\end{lemma}
\begin{proof}
Set $w =\Ll|z-z'\Rr|\identity$. For $a\in\S^\D_+$, we have $a\cdot(z+w)\geq a\cdot z+|a||z-z'| \geq a\cdot z'$, which by \eqref{e.a.b>0} implies $z+w \geqpsd z'$. Hence, for any $\eps>0$ and every $\pi\in\Pi(z')$, we can take $\pi^\eps = \pi\mathds{1}_{[0,1-\eps]}+(z+w)\mathds{1}_{(1-\eps,1]}$ to satisfy $\pi^\eps\in\Pi(z+w)$. Using this and Lemma~\ref{l.Lip_P(pi,x)}, we can obtain
\begin{align*}
     \inf_{\pi\in\Pi(z')}\sP(\pi,y) \geq \inf_{\pi\in\Pi(z+w)}\sP(\pi,y).
\end{align*}
Notice that $|\pi|\leq R+2\sqrt{D}R$ for all $\pi\in \Pi(z+w)$ and $\Ll|\pi'\Rr|\leq R$ for all $\pi'\in \Pi(z)$.
The announced result follows from this, the above display, Lemma~\ref{l.Lip_P(pi,x)} (together with the local Lipschitzness of $\nabla\xi$ and $\theta$ due to \ref{i.xi_loc_lip}) , and Lemma~\ref{e.pipi'}.
\end{proof}

\subsection{Proof of the equivalence}

We start with the following lemma.

\begin{lemma}\label{l.HJ_S^D}
It holds that
\begin{align*}\sup_{z\in\S^\D_+}\inf_{y\in\S^\D_+,\,\pi\in\Pi}\Ll\{\mathscr{P}(\pi,y) -y\cdot z+\frac{1}{2}\xi(z)\Rr\} = \sup_{z\in\S^\D}\inf_{y\in\S^\D,\,\pi\in\Pi}\Ll\{\mathscr{P}(\pi,y) -y\cdot z+\frac{1}{2}\xi(z)\Rr\}.
\end{align*}
\end{lemma}

\begin{proof}
We consider the Hopf formula
\begin{align}\label{e.Hopf_S^D}
    f(t,x) = \sup_{z\in\S^\D}\inf_{y\in\S^\D}\{\sP(y)+ z\cdot(x-y)+t\xi(z)\},\quad\forall (t,x)\in [0,\infty)\times \S^\D.
\end{align}
Since $\sP$ is convex due to Lemma~\ref{l.prop_F_N}, we can use the standard result on the Hopf formula \cite[Proposition~1]{lions1986hopf} to deduce that $f$ is the unique viscosity solution of
\begin{align*}\partial_t f- \xi(\nabla f) =0 \quad \text{on $\Ll(0,\infty\Rr)\times \S^\D$}
\end{align*}
with initial condition $f(0,\cdot)=\sP$ on $\S^\D$. For the definition of viscosity solutions, we refer to the paragraph below \cite[(5.5)]{chen2023self}.

For clarity, let us explain how to apply \cite[Proposition~1]{lions1986hopf}. We substitute $\S^\D,\, f,\, \sP,\, -\xi$ for $\R^N,\, u,\,u_0,\, H$ therein (we identify $\S^\D$ with $\R^N$ for $N= D(D+1)/2$). The assumption \cite[(8)]{lions1986hopf} holds because the Lipschitzness of $\sP$ (see Lemma~\ref{l.prop_F_N}) implies that its convex conjugate $\sP^*=+\infty$ outside a bounded set.

To proceed, we want to show that $f$ is Lipschitz and increasing. Fix any $(t,x)$ and $(t',x')$. For any $\delta>0$, we choose $z_\delta\in\S^\D$ to satisfy
\begin{align*}
    f(t,x) \leq \delta + \inf_{y\in\S^\D}\Ll\{\sP(y)+z_\delta\cdot(x-y)+t\xi(z_\delta)\Rr\}.
\end{align*}
Using~\eqref{e.y_eps_exists} and~\eqref{e.liminf<} in Lemma~\ref{l.y_eps}, we have that, for every sufficiently small $\eps\in(0,1)$, there is $y_\eps$ such that
\begin{align*}
    f(t',x') &\geq \inf_{y\in\S^\D}\Ll\{\sP(y)+z_\delta\cdot(x'-y)+t'\xi(z_\delta)\Rr\}
    \\
    & \geq -\delta + \inf_{y\in\S^\D}\Ll\{\sP(y)+\eps\sqrt{1+|y|^2}-z_\delta\cdot y\Rr\}+z_\delta \cdot x' + t'\xi(z_\delta)
    \\
    & \geq -\delta +\sP(y_\eps)+z_\delta\cdot(x'-y_\eps)+t'\xi(z_\delta).
\end{align*}
These together yield
\begin{align}\label{e.f(t,x)-f(t',x')}
    f(t,x)-f(t',x')\leq 2\delta +z_\delta\cdot(x-x')+\xi(z_\delta)(t-t').
\end{align}
Also, \eqref{e.|z-nabla_f|<eps} in Lemma~\ref{l.y_eps} implies that
\begin{align}\label{e.|z_delta-..|}
    \Ll|z_\delta - \nabla \sP(y_\eps)\Rr|\leq \eps \leq 1.
\end{align}
This along with \eqref{e.f(t,x)-f(t',x')}, the Lipschitzness of $\sP$ (see Lemma~\ref{l.prop_F_N}) and the local Lipschitzness of $\xi$ (see \ref{i.xi_loc_lip}) implies that 
\begin{align*}
    f(t,x)-f(t',x') \leq 2\delta+ C\Ll(|x-x'|+|t-t'|\Rr)
\end{align*}
for some constant $C$ depending only on $\|\sP\|_\mathrm{Lip}$ and $\xi$. Sending $\delta\to0$, we conclude that $f$ is Lipschitz.

Now, we assume that $t\leq t'$ and $x\leqpsd x'$. Recall $\nabla \sP(y_\eps)\in\S^\D_+$ due to  \eqref{e.nabla_sP>0}.
Using this, \eqref{e.a.b>0}, the first property in \ref{i.xi_sym}, and the local Lipschitzness of $\xi$, we obtain from \eqref{e.f(t,x)-f(t',x')} and \eqref{e.|z_delta-..|} that
\begin{align*}
    f(t,x) - f(t',x')&\leq 2\delta + \nabla \sP(y_\eps)\cdot (x-x') + \xi\Ll(\nabla\sP(y_\eps)\Rr)(t-t') + \eps|x-x'|+C\eps|t-t'|
    \\
    & \leq 2\delta + \eps|x-x'|+C\eps|t-t'|
\end{align*}
for some constant $C$ depending only on $\|\sP\|_\mathrm{Lip}$ and $\xi$. Sending first $\eps\to0$ and then $\delta\to0$, we obtain the following monotonicity:
\begin{align*}
    t\leq t',\, x\leqpsd x'\quad\implies \quad f(t,x)\leq f(t',x').
\end{align*}

Let $\S^\D_{++}$ be the set of positive definite matrices which is the interior of $\S^\D_+$ in $\S^\D$. By the definition of viscosity solutions, $f$ also solves
\begin{align}\label{e.PDE_cone}
    \partial_t f- \xi(\nabla f) =0 \quad \text{on $\Ll(0,\infty\Rr)\times \S^\D_{++}$}.
\end{align}
The Lipschitzness and the monotonicity proved above allow us to invoke \cite[Proposition~5.3]{chen2023self} to conclude that the restriction of $f$ to $[0,\infty)\times \S^\D_+$ is the unique viscosity solution of~\eqref{e.PDE_cone} with initial condition $f(0,\cdot) = \sP$ on $\S^\D_+$. Moreover, by \cite[Proposition~5.3]{chen2023self}, the convexity of $\sP$ gives the Hopf representation
\begin{align*}
    f(t,x) = \sup_{z\in\S^\D_+}\inf_{y\in\S^\D_+}\{\sP(y)+ z\cdot(x-y)+t\xi(z)\},\quad\forall (t,x)\in [0,\infty)\times \S^\D_+.
\end{align*}
Setting $(t,x) = (\frac{1}{2},0)$ in the above display and in \eqref{e.Hopf_S^D}, and equating the two, we obtain the desired result.
\end{proof}

Next, we prove an identity that is interesting on its own.

\begin{lemma}\label{l.infP-yz=infP-yz}
If $\xi$ is convex on $\S^D_+$, then, for every $z\in \S^\D_{+}$,
\begin{align}\label{e.infP-yz=infP-yz}
     \inf_{y\in\S^\D,\pi\in\Pi}\Ll\{\sP(\pi,y)-y\cdot z\Rr\} = \inf_{y\in \S^\D,\pi\in \Pi(z)}\Ll\{\sP(\pi,y) -y\cdot z \Rr\} .
\end{align}
\end{lemma}
\begin{proof}
Fix any $z\in\S^\D_{+}$ and we distinguish two cases.

\textit{Case~1: the left-hand side in~\eqref{e.infP-yz=infP-yz} is bounded below.}
Applying Lemma~\ref{l.y_eps} to the function $\sP$ and using~\eqref{e.P(x)=infP=infP}, for a vanishing error $o_\eps(1)$ as $\eps\to\infty$, we have that
\begin{align}
    \inf_{y\in\S^\D,\pi\in\Pi}\Ll\{\sP(\pi,y)-y\cdot z\Rr\}    &\stackrel{\eqref{e.P(x)=infP=infP}}{=}\inf_{y\in\S^\D}\Ll\{\sP(y)-y\cdot z\Rr\}\notag
    \\
    &\stackrel{\eqref{e.liminf<}}{=}\inf_{y\in \S^\D}\Ll\{\sP(y) +\eps\sqrt{1+|y|^2} -y\cdot z\Rr\} +o_\eps(1)\notag
    \\
    &\stackrel{\eqref{e.y_eps_exists}}{\geq} \sP(y_\eps) -y_\eps \cdot z +o_\eps(1)\notag
    \\
&\stackrel{\eqref{e.P(x)=infP=infP}}{=} \inf_{\pi \in \Pi(\nabla\sP(y_\eps)) }\sP(\pi, y_\eps) -y_\eps\cdot z+o_\eps(1). \label{e.inf=>>=}
\end{align}
Recall $K(z,r)$ defined in \eqref{e.K(z,r)=}.
By Lemma~\ref{l.infP>infP} and~\eqref{e.|z-nabla_f|<eps}, we get
\begin{align*}
    \inf_{\pi \in \Pi(\nabla\sP(y_\eps)) }\sP(\pi, y_\eps) - \inf_{\pi \in \Pi(z) }\sP(\pi, y_\eps)\geq -CK\Ll(z,\Ll|z-\nabla\sP(y_\eps)\Rr|\Rr) =  o_\eps(1).
\end{align*}
Inserting this to~\eqref{e.inf=>>=} and sending $\eps\to0$, we obtain
\begin{align*}
     \inf_{y\in\S^\D,\pi\in\Pi}\Ll\{\sP(\pi,y)-y\cdot z\Rr\} \geq \inf_{y\in \S^\D,\pi\in \Pi(z)}\Ll\{\sP(\pi,y) -y\cdot z \Rr\}  .
\end{align*}
Due to $\Pi\supset \Pi(z)$, the other direction trivially holds, which gives~\eqref{e.infP-yz=infP-yz}.

\textit{Case~2: the left-hand side in~\eqref{e.infP-yz=infP-yz} is equal to $-\infty$.}
We show that the right-hand side is also equal to $-\infty$.
Fix any $M>0$. 
Applying Lemma~\ref{l.-infty_y_eps} to the function $\sP$ and using~\eqref{e.P(x)=infP=infP}, we have that, for sufficiently small $\eps$,
\begin{align}
    -M &\stackrel{\eqref{e.-infty_liminf<}}{\geq}\inf_{y\in \S^\D}\Ll\{\sP(y) +\eps|y|^2-y\cdot z\Rr\}  \notag
    \\
    &\stackrel{\eqref{e.-infty_y_eps_exists}}{\geq} \sP(y_\eps) -y_\eps \cdot z \notag
    \\
    &\stackrel{\eqref{e.P(x)=infP=infP}}{=} \inf_{\pi \in \Pi(\nabla\sP(y_\eps)) }\sP(\pi, y_\eps) -y_\eps\cdot z. \label{e.-M_inf=>>=}
\end{align}
By Lemma~\ref{l.prop_F_N}, we can bound $|\nabla \sP(y_\eps)-z|\leq \|\sP\|_\mathrm{Lip}+|z|$. Hence, Lemma~\ref{l.infP>infP} gives the existence of a constant $C>0$ independent of $M$ and $\eps$ such that
\begin{align*}
    \inf_{\pi \in \Pi(\nabla\sP(y_\eps)) }\sP(\pi, y_\eps) \geq  \inf_{\pi \in \Pi(z) }\sP(\pi, y_\eps) -C.
\end{align*}
Inserting this to~\eqref{e.-M_inf=>>=}, we obtain
\begin{align*}
     -M+C \geq \inf_{y\in \S^\D,\pi\in \Pi(z)}\Ll\{\sP(\pi,y) -y\cdot z \Rr\}  .
\end{align*}
Sending $M\to\infty$, we deduce that the right-hand side in~\eqref{e.infP-yz=infP-yz} is equal to $-\infty$ and thus~\eqref{e.infP-yz=infP-yz} holds.
\end{proof}

\begin{proof}[Proof of Proposition~\ref{p.equiv}]
First, by Lemma~\ref{l.HJ_S^D}, we can work with
\begin{align*}\sup_{z\in\S^\D}\inf_{y\in\S^\D,\,\pi\in\Pi}\Ll\{\mathscr{P}(\pi,y) -y\cdot z+\frac{1}{2}\xi(z)\Rr\}.
\end{align*}
Secondly, recall the definition of $\mathcal{D}$ in \eqref{e.mathcalD} and we show that
\begin{align}\label{e.znotin=>}
    z \not\in \mathcal{D} \quad\implies  \quad \inf_{y\in \S^\D, \pi\in\Pi} \Ll\{\sP(\pi,y)-y\cdot z\Rr\} =-\infty.
\end{align}
We argue by contradiction and suppose that the infimum is strictly above $-\infty$. 
Applying \eqref{e.|z-nabla_f|<eps} in Lemma~\ref{l.y_eps} to $\sP$, we have that $z\in K$ for $K=\overline{\{ \nabla \sP(x):\:x\in\S^\D\}}$. Using $\frac{\sigma\sigma^\intercal}{N}=\sum_{i=1}^N\frac{1}{N}\sigma_{\cdot,i}\sigma_{\cdot,i}^\intercal$ and recalling that each column vector $\sigma_{\cdot,i}\in \supp P_1$ a.s., we have $\frac{\sigma\sigma^\intercal}{N} \in \mathcal{D}$. In view of Theorem~\ref{t}~\eqref{i.t.cvg_self-overlap}, we have that $K\subset \mathcal{D}$ and thus $z\in\mathcal{D}$ reaching a contradiction.

Therefore, we have verified \eqref{e.znotin=>} which implies that
\begin{align*}\sup_{z\in\S^\D}\inf_{y\in\S^\D,\,\pi\in\Pi}\Ll\{\mathscr{P}(\pi,y) -y\cdot z+\frac{1}{2}\xi(z)\Rr\} = \sup_{z\in\mathcal{D}}\inf_{y\in\S^\D,\,\pi\in\Pi}\Ll\{\mathscr{P}(\pi,y) -y\cdot z+\frac{1}{2}\xi(z)\Rr\}.
\end{align*}
Notice that $\mathcal{D}\subset \S^\D_+$ by the definition of $\mathcal{D}$ in \eqref{e.mathcalD}.
The desired result~\eqref{e.equiv_formulae} follows from the above display and Lemma~\ref{l.infP-yz=infP-yz}.
\end{proof}

\appendix

\section{Rewriting Panchenko's formula}\label{appendix}

We explain how to rewrite the generalized Parisi formula obtained by Panchenko in \cite{pan.vec} into \eqref{e.pan.parisi}. 
To resolve notational conflicts, we relabel $\kappa,\, \Gamma^\kappa;\; D,\, \Pi_D;\;\xi,\, \theta$ in \cite[Section~1]{pan.vec} by $\D,\, \S^\D;\; z,\,\Pi(z);\; \bxi,\, \btheta$, respectively.
We emphasize that $D$ in \cite{pan.vec} (relabeled by $z$ here) is the endpoint of a monotone path $\pi$, while here $D$ is the dimension.

In \cite[(1)--(3)]{pan.vec}, the notation of the spin configuration is given as follows: $\sigma_i$ therein is the $i$-th column vector $\sigma_{\cdot,i}$ here ($1\leq i\leq N$); $\sigma(k)$ therein is the $k$-th row vector $\sigma_{k,\cdot}$ here ($1\leq k\leq D$). In particular, $\sigma_i(k)$ is $\sigma_{k,i}$ here, and we have
\begin{align*}
    \sigma_{i}(k)\sigma'_{i}(k') = \Ll(\sigma_{\cdot,i}\sigma'^\intercal_{\cdot,i}\Rr)_{k,k'},\quad \sum_{i=1}^N \sigma_{i}(k)\sigma'_{i}(k') = \Ll(\sigma\sigma'^\intercal\Rr)_{k,k'}
\end{align*}

In \cite[(4)--(6)]{pan.vec}, the Hamiltonian $H_N(\sigma)$ is defined in the fashion of mixed $p$-spin interactions. We can summarize \cite[(9)--(12)]{pan.vec} as
\begin{align*}
    \E H_N(\sigma)H_N(\sigma') = N\sum_{k,k'=1}^\D \bxi_{k,k'}\Ll(\frac{\Ll(\sigma\sigma'^\intercal\Rr)_{k,k'}}{N}\Rr)
\end{align*}
where each $\bxi_{k,k'}:\R\to\R$ is given in \cite[(9)]{pan.vec}.
Comparing with our definition of $\xi$ in~\eqref{e.xi}, we have
\begin{align}\label{e.xi(a)=sum_bxi}
    \xi(a) = \sum_{k,k'}^\D \bxi_{k,k'}(a_{k,k'}),\quad\forall a\in\R^{\D\times\D}.
\end{align}

In \cite[(13)]{pan.vec}, $\mathcal{D}$ is defined in the same way as in~\eqref{e.mathcalD}. The definition of $\Pi$ in \cite[(15)]{pan.vec} is the same as in~\eqref{e.Pi}.
In \cite[(16)]{pan.vec}, the collection $\Pi(z)$ of paths with pinned endpoint is defined, with the difference that there $\pi\in \Pi(z)$ is required to satisfy $\pi(0)=0$ and $\pi(1)=z$. Here, we only require $\pi(0)\geqpsd 0$ and $\pi(1)= z$. The difference is negligible due to approximation and Lemma~\ref{l.Lip_P(pi,x)}.

In \cite[(17)--(19)]{pan.vec}, discrete paths are defined. They are of the following form: for some $r\in\N$,
\begin{align}\label{e.pi(s)=}
    \pi(s) = \sum_{j=0}^r \gamma_j \mathds{1}_{(x_{j-1},x_j]},\quad \forall s\in [0,1]
\end{align}
for a sequence of increasing real numbers
\begin{align*}
    x_{-1}=0\leq x_0\leq \cdots \leq x_{r-1}\leq x_r =1
\end{align*}
and a sequence of increasing matrices in $\S^\D_+$
\begin{align*}
    0=\gamma_0\leqpsd \gamma_1\leqpsd \cdots\leqpsd \gamma_{r-1}\leqpsd\gamma_r=z.
\end{align*}

In \cite[(20)--(22)]{pan.vec}, $\btheta$ and $\bxi'$ are defined. We summarize them here. For each $k,k'$, let $\bxi'_{k,k'}$ be the derivative of $\bxi_{k,k'}$, and set
\begin{align*}
    \btheta_{k,k'}(x) = x\bxi'_{k,k'}(x) - \bxi_{k,k'}(x),\quad\forall x\in\R.
\end{align*}
Then, for $a\in \R^{\D\times \D}$, define
\begin{align*}
    \bxi(a) = \Ll(\bxi_{k,k'}(a_{k,k'})\Rr)_{k,k'},\  \bxi'(a) = \Ll(\bxi'_{k,k'}(a_{k,k'})\Rr)_{ k,k'},\  \btheta(a) = \Ll(\btheta_{k,k'}(a_{k,k'})\Rr)_{ k,k'}.
\end{align*}
In view of~\eqref{e.xi(a)=sum_bxi}, we have the following relation
\begin{align}\label{e.bxi'=nabla_xi}
    \bxi' =\nabla \xi.
\end{align}
Later in \cite[(27)]{pan.vec}, the following notation is introduced
\begin{align*}
    \mathrm{Sum}(a) =\sum_{k,k'=1}^\D a_{k,k'},\quad\forall a \in \R^{\D\times \D}.
\end{align*}
Recall the definition of $\theta$ in~\eqref{e.theta}.
Using the above, we can easily verify that
\begin{align}\label{e.sumtheta=theta}
    \mathrm{Sum}\Ll(\btheta(a)\Rr) = a \cdot \nabla \xi(a) - \xi(a)=\theta(a),\quad\forall a \in \R^{\D\times \D}.
\end{align}

In \cite[(23)--(26)]{pan.vec}, the first term $\Phi$ in the Parisi functional is defined, where a parameter $\lambda$ is needed. There, $\lambda=\Ll(\lambda_{k,k'}\Rr)_{1\leq k\leq k'\leq \D}\in \R^{\D(\D+1)/2}$. Here, we identify $\lambda$ with the symmetric matrix $\tilde \lambda \in \S^\D$ given by $\tilde \lambda_{k,k}=\lambda_{k,k}$ and $\tilde \lambda_{k,k'}=\frac{1}{2}\lambda_{k,k'}$ for $k<k'$. Then, for every $a\in\S^\D$, we have
\begin{align}\label{e.sumlambda=lambdadot}
    \sum_{k\leq k'}\lambda_{k,k'}a_{k,k'} = \lambda\cdot a.
\end{align}
Using this identification,~\eqref{e.bxi'=nabla_xi}, and the standard computation of the Ruelle probability cascade (see the proof of \cite[Lemma~3.1]{pan}), we can rewrite $\Phi$ in \cite{pan.vec} as
\begin{align*}
    \Phi(\lambda,z,\pi) = \E \log \int\exp\Ll(w^{\nabla\xi\circ\pi}(\alpha)\cdot \sigma + \lambda\cdot \sigma\sigma^\intercal\Rr) \d P_1(\sigma)\d \mathfrak{R}(\alpha).
\end{align*}
Here, we have simplified the notation as explained below \cite[(28)]{pan.vec} by absorbing the dependence of $(x_j)_{j=0}^r$ and $(\gamma_j)_{j=0}^r$ to $\pi$.

In \cite[(27)--(31)]{pan.vec}, the definition of the Parisi functional is given and comments on its extension to continuous paths are made. Using~\eqref{e.sumtheta=theta} and~\eqref{e.sumlambda=lambdadot}, we can express the functional $\mathcal{P}$ defined in \cite[(31)]{pan.vec} as
\begin{align*}
    \mathcal{P}(\lambda,z,\pi) = \Phi(\lambda,z,\pi)-\lambda\cdot z - \frac{1}{2}\theta(z)+\frac{1}{2}\int_0^1 \theta(\pi(s))\d s.
\end{align*}
Comparing this with~\eqref{e.sP(pi,x)}, we have
\begin{align*}
    \Phi(\lambda,z,\pi) +\frac{1}{2}\int_0^1 \theta(\pi(s))\d s = \sP\Ll(\pi,\lambda + \frac{1}{2}\nabla\xi\circ\pi(1)\Rr).
\end{align*}
Recall the constraint $\pi(1)=z$ in \eqref{e.pi(s)=}.
Using this and the definition of $\theta$ in \eqref{e.theta}, we have
\begin{align}\label{e.P=P.....}
    \mathcal{P}(\lambda,z,\pi) = \sP\Ll(\pi,\lambda + \frac{1}{2}\nabla\xi(z)\Rr) - \Ll(\lambda + \frac{1}{2}\nabla\xi(z)\Rr)\cdot z + \frac{1}{2}\xi(z).
\end{align}

The generalized Parisi formula given in \cite[(32) in Theorem~1]{pan.vec} is
\begin{align*}
    \lim_{N\to\infty} F^\std_N = \sup_{z\in \mathcal{D}}\inf_{\lambda\in\S^\D,\,\pi\in\Pi(z) }\mathcal{P}(\lambda,z,\pi).
\end{align*}
Since $\lambda$ range over the linear space $\S^\D$, we can translate $\lambda$ and change $\lambda+\frac{1}{2}\nabla\xi(z)$ to a variable $y$. This along with~\eqref{e.P=P.....} gives
\begin{align*}
    \lim_{N\to\infty} F^\std_N = \sup_{z\in \mathcal{D}}\inf_{y\in\S^\D,\,\pi\in\Pi(z) }\Ll\{\sP(\pi,y)-y\cdot z +\frac{1}{2}\xi(z)\Rr\},
\end{align*}
which is the formula in~\eqref{e.pan.parisi}.

\small
\bibliographystyle{abbrv}
\newcommand{\noop}[1]{} \def\cprime{$'$}

\end{document}